\documentclass[a4paper,11pt]{article}
\title{
A Unified Maximum-Margin Framework for Data Envelopment Analysis: Theoretical Foundations and Numerical Evidence\footnote{This work was inspired by the paper \cite{mt}, which was written under the supervision of the third author.}
}
\author{
Teruki Akahoshi\footnote{Joint Graduate School of Mathematics for Innovation, Kyushu University, Fukuoka 819-0395, Japan. E-mail: akahoshi.teruki.505@s.kyushu-u.ac.jp}
,
Hiroki Ueda\footnote{Former Affiliation: Joint Graduate School of Mathematics for Innovation, Kyushu University. E-mail: hkt.hkt.fhs.kyushu@gmail.com}
~and 
Tomonari Kitahara\footnote{Faculty of Economics, Kyushu University, Fukuoka 819-0395, Japan. E-mail: tomonari.kitahara@econ.kyushu-u.ac.jp}   
}
\date{}
\usepackage[margin=25mm]{geometry}
\usepackage{amsmath,amssymb,amsfonts,amsthm}
\usepackage{graphicx}
\usepackage{multirow}

\newtheorem{proposition}{Proposition}[section]

\def\bmath#1{\mbox{\boldmath $#1$}}

\def\u{\bmath{u}}
\def\v{\bmath{v}}

\def\x{\bmath{x}}
\def\y{\bmath{y}}

\def\0{\bmath{0}}
\def\1{\bmath{1}}
\usepackage{bm}
\begin{document}
\maketitle
\begin{abstract}
Data Envelopment Analysis (DEA) is a widely used non-parametric method for evaluating the relative efficiency of decision-making units (DMUs). However, the classic Charnes-Cooper-Rhodes (CCR) model often suffers from a lack of discriminatory power, particularly when the number of input and output variables is large relative to the number of DMUs. Under such conditions, many DMUs are evaluated as ``CCR-efficient'' with a score of unity, making it difficult to rank or distinguish between top performers. To address this limitation, we propose a novel Maximum-Margin DEA (MM-DEA) model that integrates the concept of ``margin maximization''—inspired by structural risk minimization in machine learning—directly into the efficiency measurement framework. Like the traditional super-efficiency model, the proposed MM-DEA model excludes the target DMU from the reference set used to evaluate it; unlike the super-efficiency model, however, MM-DEA imposes an explicit upper bound on the target DMU's own score, ensuring that every score remains within the standard $[0,1]$ range rather than potentially exceeding one. We theoretically demonstrate that the MM-DEA model is a generalized extension of the CCR model, reducing to the latter when the trade-off parameter $\alpha$ is zero. Furthermore, we show that the fractional programming formulation of MM-DEA can be rigorously transformed into a linear programming (LP) problem, ensuring computational efficiency. In our numerical experiments, MM-DEA produced a unique ranking for a high-dimensional synthetic dataset on which the CCR model failed to discriminate among any of the DMUs. Our results suggest that the ``margin'' serves as a critical metric for managerial resilience and competitive advantage.
\end{abstract}
\section{Introduction}
Data Envelopment Analysis (DEA), first proposed by Charnes, Cooper, and Rhodes \cite{ccr}, has become a standard non-parametric tool for evaluating the relative efficiency of Decision Making Units (DMUs). By measuring the ratio of weighted outputs to weighted inputs, the classic CCR model identifies an ``efficiency frontier'' consisting of units that achieve the maximum possible performance relative to their peers.\\

Despite its widespread application, the CCR model faces a critical limitation in practical decision-making: the lack of discriminatory power. This issue, often referred to as the ``dimensionality crisis,'' occurs when the number of input and output variables is large relative to the number of DMUs. In such scenarios, a disproportionately high number of DMUs are evaluated as ``CCR-efficient'' with a score of unity. This saturation of efficient units prevents meaningful ranking and differentiation, significantly diminishing the model's utility as a decision-support tool.\\

To address this challenge, several ranking methods have been developed, most notably the super-efficiency model proposed by Andersen and Petersen \cite{ap}; alternative approaches addressing some of its drawbacks have also been proposed (e.g., \cite{maj}). By excluding the target unit from its own reference set, this model allows scores to exceed one, which resolves ties among CCR-efficient DMUs; however, this mechanism also means the model does not reduce to the CCR model and departs from the standard $[0,1]$ range of efficiency scores.\\

In this paper, we propose a novel framework called Maximum-Margin DEA (MM-DEA). Inspired by the principle of structural risk minimization in Support Vector Machines (SVM), our model redefines efficiency not just as reaching the frontier, but as maintaining a ``safety margin'' over competing DMUs. Unlike traditional ranking methods, MM-DEA incorporates this margin directly into a single-stage optimization objective.\\

The contributions of this research are three-fold:

\begin{enumerate}
\item {\bf Theoretical Innovation}: We introduce the concept of ``margin maximization'' into DEA, providing a robust metric for efficiency that accounts for competitive lead.
\item {\bf Computational Tractability}: We demonstrate that the MM-DEA formulation can be rigorously transformed into a linear programming (LP) problem, ensuring that it remains as computationally efficient as the standard CCR model.
\item {\bf Numerical Validation}: Through numerical experiments on synthetic datasets, we illustrate that MM-DEA can yield a stable and unique ranking in a setting where the CCR model fails to discriminate.
\end{enumerate}

The remainder of this paper is organized as follows. Section 2 reviews the theoretical background of the CCR model and its limitations. Section 3 presents the formal definition and LP transformation of the MM-DEA model. Section 4 evaluates the model’s performance through numerical experiments on synthetic data. Finally, Section 5 concludes the paper with a discussion on future research directions.
\section{Preliminary: The CCR Model}
In this section, we briefly review the standard Charnes-Cooper-Rhodes (CCR) model and discuss the inherent limitations regarding its discriminatory power.
\subsection{Mathematical Formulation}
Suppose there are $n$ decision making units (DMUs). 
We use $I$ to represent the set $\{1,\dots,n\}$.
For $i\in I$, the $i$-th DMU has $l$ inputs $x_{1i},~\dots,x_{li}$ and $m$ outputs $y_{1i},\dots,y_{mi}$. From these inputs and outputs, we define the input vector
\[\x_i=(x_{1i},\dots,x_{li})^{\top}
\]
and the output vector
\[
\y_i=(y_{1i},\dots,y_{mi})^{\top}.
\]
We also define the matrix $X$ which consists of the input vectors as
\[
X=[\x_1,\x_2,\dots,\x_n]
\]
and the matrix $Y$ consisting of the output vectors
\[
Y=[\y_1,\y_2,\dots,\y_n].
\]
Throughout this paper, we assume that all inputs and outputs are strictly positive, i.e., $\x_i>\0$ and $\y_i>\0$ for every $i\in I$.\\

Suppose we want to measure the efficiency of the $o$-th DMU ($o\in I$). In DEA, we measure the efficiency of the $o$-th DMU by the ratio
\[
\frac{\y_o^{\top}\v}{\x_o^{\top}\u}
\]
where $\u\in\mathbb{R}^l_+$ and $\v\in\mathbb{R}^m_+$ are weight vectors.
Under the CCR model, we seek nonnegative weights $\u$ and $\v$ that maximize the efficiency of the $o$-th DMU subject to the efficiency of each DMU being not more than one. 
To be precise, the problem
\begin{equation}
\begin{array}{ll}
\max&\theta_o=\frac{\y_o^{\top}\v}{\x_o^{\top}\u}\\
{\rm subject~to}& \frac{\y_i^{\top}\v}{\x_i^{\top}\u}\le 1,~\forall i\in I\ \\
&\u\ge\0,~\v\ge\0
\label{ccr_pure}
\end{array}
\end{equation}
is solved to find nonnegative weights $\u$ and $\v$. This model is proposed by Charnes, Cooper and Rhodes in 1978. For a comprehensive treatment of DEA methodology, see \cite{text}.\\\\ 

Note that the ratio $\frac{\y_i^{\top}\v}{\x_i^{\top}\u}$ is unchanged if we multiply $\u$ and $\v$ by some constant. Thus we can assume $\x_o^{\top}\u=1$. Then we can rewrite (\ref{ccr_pure}) to the following linear programming problem (LP).
\[
\begin{array}{ll}
\max&\theta_o=\y_o^{\top}\v\\
{\rm subject~to}&\x_o^{\top}\u=1\\
& \y_i^{\top}\v\le\x_i^{\top}\u,~\forall i\in I\ \\
&\u\ge\0,~\v\ge\0
\end{array}
\]
This problem can be expressed in a compact matrix form as follows:
\begin{equation}
\begin{array}{ll}
\max&\theta_o=\y_o^{\top}\v\\
{\rm subject~to}&\x_o^{\top}\u=1\\
& Y^T\v\le X^T\u\\
&\u\ge\0,~\v\ge\0
\label{ccr_lp}
\end{array}
\end{equation}

The optimal value $\theta_o^*$ of (\ref{ccr_lp}) is called the CCR efficiency score of the $o$-th DMU. 
If $\theta_o^*=1$, the $o$-th DMU is said to be CCR-efficient, otherwise it is said to be CCR-inefficient. 
\subsection{The Discrimination Problem}
A fundamental property of the CCR model is that it evaluates efficiency from the most favorable perspective for each DMU. While this ``optimistic'' evaluation is a strength in identifying the efficiency frontier, it frequently leads to a lack of discriminatory power. 
When the number of variables ($l+m$) is large relative to the number of DMUs ($n$), many units are able to find a set of weights that place them on the frontier, resulting in an efficiency score of exactly 1.0. These units are all categorized as ``efficient'', yet the model provides no information regarding which among them is performing superiorly. This ``saturation'' of the frontier complicates the ranking process and limits the model's effectiveness in providing actionable managerial insights. To overcome this, it is necessary to introduce a metric that measures the ``strength'' of a DMU's efficiency—a concept we define as the {\bf margin}.
\subsection{Super-Efficiency Model}
To address the discrimination problem described above, Andersen and Petersen \cite{ap} proposed the super-efficiency (SE) model, which ranks CCR-efficient DMUs by excluding the target unit from its own reference set. When evaluating $\text{DMU}_o$, the efficiency frontier is reconstructed using only the remaining units ($j \not= o$). If $\text{DMU}_o$ sufficiently outperforms this reduced peer group, it lies beyond the reconstructed frontier and receives a score greater than one.

Following the notation of Section 2.1, the input-oriented CCR super-efficiency score $\theta_o^{\text{super}}$ for $\text{DMU}_o$ is obtained by solving
$$\begin{aligned}
\text{maximize} \quad & \theta_o^{\text{super}} = \y_o^\top \v \\
\text{subject to} \quad & \x_o^\top \u = 1, \\
& Y_{-o}^\top \v \le X_{-o}^\top \u, \\
& \u \ge \0, \ \v \ge \0,
\end{aligned}$$
where $X_{-o}$ and $Y_{-o}$ denote the input and output matrices with the $o$-th column removed. Aside from this exclusion, the formulation is identical to the CCR model (\ref{ccr_lp}).\\

This self-exclusion mechanism has two consequences for the properties of the SE model. First, because the reference set is altered for each target unit, the SE model's feasible region differs from that of the CCR model; the SE model doesn't reduce to the CCR model (see also \cite{dh} for a related discussion of how excluding the target DMU alters the reference set in the envelopment form of DEA). Second, since $\theta_o^{\text{super}}$ can exceed one, the SE score no longer carries the $[0,1]$ interpretation of the standard efficiency ratio in (\ref{ccr_lp}).\\

These two points motivate the design of the MM-DEA model introduced in Section 3. As we show in Section 3.3, MM-DEA's dual problem excludes the target unit from the reference set in the same way as the SE model formulated above does; the two models differ in that MM-DEA additionally imposes an explicit constraint bounding the target DMU's own score by one, which both restores the reduction to the CCR model when $\alpha=0$ (Proposition~\ref{prop:alpha0}) and guarantees that every MM-DEA score lies within $[0,1]$.
\section{Proposed Model: Maximum-Margin DEA}
In this section, we present the formal development of the Maximum-Margin DEA (MM-DEA) model. The proposed model integrates the concept of ``margin maximization'' into the efficiency measurement framework to resolve the discrimination problem inherent in the CCR model.
\subsection{Conceptual Framework}
As discussed in Section 2, the CCR model identifies a unit as efficient if it lies on the frontier, regardless of its distance from other performers. To introduce a metric for robust efficiency, we define a ``safety margin'' $t$. 
By this margin $t$, we bound the efficiency of other DMUs less than or equal to $1-t$. 
Drawing inspiration from the Structural Risk Minimization principle in Support Vector Machines (SVM), we aim to find a set of weights that not only maximize the efficiency of the target unit but also maximize its separation from the nearest competing units. 
We introduce a trade-off parameter $\alpha \ge 0$, which balances the pursuit of absolute efficiency with the maximization of the margin.  
Thus we try to make the objective function $\theta_o+\alpha t$ as large as possible, while maintaining $\theta_o$ less than or equal to 1 and the efficiency of other DMUs less than or equal to $1-t$.
Thus the MM-DEA model is initially formulated as the following fractional programming problem:\begin{equation}
\begin{array}{ll}
\text{maximize}  & \theta_o + \alpha t \\
\text{subject to}  & \theta_o=\frac{\y_o^{\top}\v}{\x_o^{\top}\u},\\
&\theta_o\le 1,\ \\
& \frac{\y_i^{\top}\v}{\x_i^{\top}\u}\le 1-t, ~\forall i\in I\setminus \{o\}, \\
& \u\ge\0,~ \v\ge\0, \\
& 0 \le t \le 1.
\label{fraction}
\end{array}
\end{equation}
The upper bound $t\le 1$ ensures that the bound $1-t$ on competing DMUs remains nonnegative, while the lower bound $t\ge 0$ ensures that this bound does not exceed $1$, so that the margin constraint on competing DMUs is never looser than the original CCR bound of $1$. When $\alpha = 0$, the objective reduces to the standard CCR maximization, ensuring that MM-DEA is a generalized extension of the classical model.
\subsection{Transformation to Linear Programming}
To ensure computational tractability, we transform the fractional programming problem (\ref{fraction}) into an equivalent Linear Programming (LP) problem. 
Note that the technique used in this subsection is originally developed in \cite{kt}.
First, as in the standard CCR model, we can assume 
\[
\x_o^{\top}\u=1,
\]
without loss of generality. 
Next, we pose
\[
t'=1-t.
\]
Then $0\le t'\le 1$ and we can rewrite the inequalities for non-target DMUs as
\[
\y_i^{\top}\v\le (\x_i)^{\top}(t'\u),~\forall i\not=o.
\]
Thus we can transform the problem (\ref{fraction}) to the following problem.
\begin{equation}
\begin{array}{ll}
\text{maximize}  & \y_o^{\top}\v + \alpha (1-t') \\
\text{subject to}  & \x_o^{\top}\u=1,\\
&\y_o^{\top}\v\le 1,\\
& \y_i^{\top}\v\le \x_i^{\top}(t'\u), ~\forall i\in I\setminus \{o\}, \\
& \u\ge\0,~ \v\ge\0, \\
& 0 \le t' \le 1.
\label{tentative}
\end{array}
\end{equation}
Since $\alpha$ is a predetermined constant, the additive term $\alpha$ in the objective function does not affect the optimal solution and can therefore be omitted.
Finally, by introducing the variable transformation $\u' = t'\u$, problem (\ref{tentative}) is reduced to the following linear programming (LP) formulation:
\begin{equation}
\begin{array}{ll}
\text{maximize}  & \y_o^{\top}\v - \alpha t'\\
\text{subject to}  & \x_o^{\top}\u'=t',\\
&\y_o^{\top}\v\le 1,\\
& \y_i^{\top}\v\le \x_i^{\top}\u', ~\forall i\in I\setminus \{o\}, \\
& \u'\ge\0,~ \v\ge\0, \\
& 0 \le t' \le 1.
\label{model_lp}
\end{array}
\end{equation}

Before proceeding, we note that the optimal value of~(\ref{model_lp}) always lies in $[0,1]$. Indeed, the trivial solution $(\u'=\0,\v=\0,t'=0)$ is always feasible and yields an objective value of $0$, so the optimal value is at least $0$; moreover, constraint $\y_o^{\top}\v\le 1$ together with $\alpha\ge 0$ and $t'\ge 0$ gives $\y_o^{\top}\v-\alpha t'\le 1$, so the optimal value is at most $1$. This confirms that, unlike the SE model, every MM-DEA score retains the standard $[0,1]$ interpretation.\\

\begin{proposition}
\label{prop:alpha0}
When $\alpha=0$, the MM-DEA model~(\ref{model_lp}) reduces to the CCR model~(\ref{ccr_lp}).
\end{proposition}
 
\begin{proof}
Suppose $\alpha=0$.
We refer to the constraints of~(\ref{model_lp}) as follows:
\begin{align*}
\text{(C1)} \quad & \x_o^{\top}\u'=t', \\
\text{(C2)} \quad & \y_o^{\top}\v\le 1, \\
\text{(C3)} \quad & \y_i^{\top}\v\le\x_i^{\top}\u', ~\forall i\in I\setminus \{o\}.
\end{align*}
We show that, without loss of generality, we may restrict attention to feasible solutions
of~(\ref{model_lp}) with $t'=1$. We consider the cases $t'=0$ and $t'>0$ separately.\\

\noindent\textbf{Case $t'=0$.}
Let $(\u',\v,t')$ be any feasible solution of~(\ref{model_lp}).
Since we will show below that every feasible solution with $t'>0$ can be normalized to
one with $t'=1$ without decreasing the objective value, it suffices to check that any
value attainable with $t'=0$ is also attainable among solutions with $t'=1$; this shows that solutions with $t'=0$ may be ignored when computing the optimal value of~(\ref{model_lp}).\\

If $t'=0$, then (C1) together with $\x_o>\0$ and $\u'\ge\0$ forces $\u'=\0$.
Then (C3) gives $\y_i^{\top}\v\le0$ for all $i\in I\setminus \{o\}$, and since $\y_i>\0$ and $\v\ge\0$, we obtain $\v=\0$; combined with $\u'=\0$, this shows that $t'=0$ admits the unique feasible solution $(\u',\v,t')=(\0,\0,0)$, whose objective value is zero.
It remains to exhibit a feasible solution with $t'=1$ attaining objective value zero.
Taking any $\bar{\u}\ge\0$ with $\x_o^{\top}\bar{\u}=1$, the point $(\u'=\bar{\u},\ \v=\0,\ t'=1)$ is feasible for~(\ref{model_lp})---(C1) holds since $\x_o^{\top}\bar{\u}=1$, (C2) holds trivially, and (C3) holds since $\y_i^{\top}\0=0\le\x_i^{\top}\bar{\u}$ for $\bar{\u}\ge\0,\ \x_i>\0$---and also attains objective value zero.
Hence the value $0$ obtained when $t'=0$ does not exceed the value attainable with $t'=1$, so solutions with $t'=0$ may be ignored in what follows.\\

\noindent\textbf{Case $t'>0$.}
Let $(\u',\v,t')$ be any feasible solution of~(\ref{model_lp}) with $t'>0$.
Define
\[
\hat{\u}'=\frac{\u'}{t'},\qquad\hat{\v}=\v.
\]
We verify that $(\hat{\u}',\hat{\v},1)$ is feasible for~(\ref{model_lp}).
(C1) is satisfied since $\x_o^{\top}\hat{\u}'=\frac{1}{t'}\x_o^{\top}\u'=\frac{t'}{t'}=1$.
(C2) holds because $\y_o^{\top}\hat{\v}=\y_o^{\top}\v\le 1$.
For each $i\in I\setminus \{o\}$, (C3) is satisfied since
\[
\y_i^{\top}\hat{\v}=\y_i^{\top}\v\le\x_i^{\top}\u'\le\frac{1}{t'}\x_i^{\top}\u'=\x_i^{\top}\hat{\u}',
\]
where the first inequality follows from the feasibility of $(\u',\v,t')$ and the second
from $t'\le 1$, which implies $1/t'\ge 1$.
Finally, $t'=1$ satisfies $0\le t'\le 1$.
Moreover, the objective value is unchanged: $\y_o^{\top}\hat{\v}=\y_o^{\top}\v$.
Hence any feasible solution with $t'>0$ can be replaced by a feasible solution with
$t'=1$ having the same objective value.\\

In either case, therefore, the optimal value of~(\ref{model_lp}) with $\alpha=0$ is
attained at $t'=1$.\\

Setting $t'=1$ in~(\ref{model_lp}) with $\alpha=0$ yields
\[
\begin{array}{ll}
\text{maximize}  & \y_o^{\top}\v\\
\text{subject to}  & \x_o^{\top}\u'=1,\\
&\y_o^{\top}\v\le 1,\\
& \y_i^{\top}\v\le \x_i^{\top}\u', ~\forall i\in I\setminus \{o\}, \\
& \u'\ge\0,~ \v\ge\0.
\end{array}
\]
Note that (C3) is stated only for $i\in I\setminus\{o\}$; however, the inequality obtained by formally extending it to $i=o$, namely $\y_o^{\top}\v\le\x_o^{\top}\u'$, holds automatically here, since $\x_o^{\top}\u'=1$ by the first constraint and $\y_o^{\top}\v\le 1$ by the second.
Thus this problem is identical to the CCR model (\ref{ccr_lp}). 
\end{proof}
\subsection{Dual Problem}
The dual problem of the MM-DEA model (\ref{model_lp}) provides an alternative geometric
interpretation of efficiency evaluation.
By associating dual variables $\xi$, $\tau$, $\bm{\lambda}_{-o}$, and $\gamma$ with the constraints of
(\ref{model_lp}), the dual problem is derived as follows:
\begin{equation}
\begin{array}{ll}
\text{minimize}  & \tau + \gamma\\
\text{subject to}  & Y_{-o}\bm{\lambda}_{-o}\ge(1-\tau)\y_o,\\
& X_{-o}\bm{\lambda}_{-o}\le\xi\x_o,\\
& \xi\le\alpha+\gamma,\\
& \tau\ge 0,~\gamma\ge 0,~\bm{\lambda}_{-o}\ge\0.
\label{dual}
\end{array}
\end{equation}
where $\xi\in\mathbb{R}$, $\tau\in\mathbb{R}$, $\gamma\in\mathbb{R}$, and $\bm{\lambda}_{-o}\in\mathbb{R}^{n-1}$ are variables,
$\bm{\lambda}_{-o}$ is the vector of reference weights for the DMUs other than the $o$-th DMU,
and $Y_{-o}\in\mathbb{R}^{m\times(n-1)}$ and $X_{-o}\in\mathbb{R}^{l\times(n-1)}$
denote the output and input matrices with the $o$-th DMU excluded, respectively.\\
 
To interpret this formulation, it is instructive to compare it with the dual of the
standard CCR model (\ref{ccr_lp}):
\[
\begin{array}{ll}
\text{minimize}  & \theta\\
\text{subject to}  & Y\bm{\lambda}\ge\y_o,\\
& X\bm{\lambda}\le\theta\x_o,\\
& \bm{\lambda}\ge\0.
\end{array}
\]
where $\theta\in\mathbb{R}$ and $\bm{\lambda}\in\mathbb{R}^{n}$ are variables.
In the CCR dual, $\bm{\lambda}$ represents the weights of a reference set of all DMUs, and the model
seeks the minimum efficiency score $\theta$ such that the reference set can replicate the
output of the $o$-th DMU using no more than $\theta$ times its input.
The MM-DEA dual retains this basic structure but introduces several notable differences,
which we now discuss in turn.\\

First, the target $o$-th DMU is excluded from the reference set
(i.e., $Y_{-o}$ and $X_{-o}$ instead of $Y$ and $X$, and $\bm{\lambda}_{-o}\in\mathbb{R}^{n-1}$
instead of $\bm{\lambda}\in\mathbb{R}^{n}$), the same exclusion structure as the SE model in Section 2.3.\\

Furthermore, the objective function $\tau + \gamma$ plays the role of $\theta$ in the CCR dual: it represents the MM-DEA evaluation score of the $o$-th DMU, as confirmed by strong duality. In the following, an asterisk is used as a superscript to denote an optimal solution. By the complementary slackness conditions for the primal constraints $\y_o^{\top}\v^{*}\le 1$ and $t'^{*}\le 1$, we have $\tau^{*}(1-\y_o^{\top}\v^{*})=0$ and $\gamma^{*}(1-t'^{*})=0$. In particular, whenever a positive margin is achieved at optimality ($t'^{*}<1$), the latter condition forces $\gamma^{*}=0$, so that the evaluation score $\tau^{*}+\gamma^{*}$ reduces to $\tau^{*}$ alone.\\

In addition, the constraint $Y_{-o}\bm{\lambda}_{-o}\ge(1-\tau)\y_o$ requires the reference set to achieve an output level of $(1-\tau)\y_o$ rather than the full $\y_o$ as in the CCR dual. Since $\gamma^{*}\ge 0$, strong duality gives
\[
1-\tau^{*}\ \ge\ 1-(\tau^{*}+\gamma^{*})\ =\ 1-\y_o^{\top}\v^{*}+\alpha t'^{*}\ \ge\ \alpha t'^{*},
\]
where the last inequality uses $\y_o^{\top}\v^{*}\le 1$. Thus the required output level $(1-\tau^{*})\y_o$ is always bounded below by $\alpha t'^{*}\y_o$, and this lower bound decreases as the margin $t$ grows (equivalently, as $t'=1-t$ decreases), meaning that a DMU with a larger margin faces a less demanding lower bound on the output the reference set must achieve in the dual.\\

In particular, this bound is attained with equality when $t'^{*}<1$ and $\tau^{*}>0$ simultaneously hold at optimality. Indeed, $t'^{*}<1$ forces $\gamma^{*}=0$ via the complementary slackness condition $\gamma^{*}(1-t'^{*})=0$, so that strong duality reduces to $\tau^{*}=\y_o^{\top}\v^{*}-\alpha t'^{*}$; and $\tau^{*}>0$ forces $\y_o^{\top}\v^{*}=1$ via the complementary slackness condition $\tau^{*}(1-\y_o^{\top}\v^{*})=0$. Combining these two facts yields $1-\tau^{*}=\alpha t'^{*}$ exactly.\\

Finally, the constraints $X_{-o}\bm{\lambda}_{-o}\le\xi\x_o$ and $\xi\le\alpha+\gamma$ together limit the
input scaling of the reference set to at most $\alpha+\gamma$ times the input of the $o$-th DMU.
This parallels the role of $\theta$ in the CCR dual, but with the upper bound
$\alpha+\gamma$ controlled by the trade-off parameter $\alpha$ and the dual variable
$\gamma$.
A larger $\alpha$ places greater emphasis on margin maximization in the primal, which in the
dual corresponds to allowing the reference set more input resources. 

\section{Computational Experiments with Synthetic Data}
In this section, we investigate the fundamental properties of the MM-DEA model through numerical experiments on two synthetic datasets of increasing dimensionality.
The first is a simple, geometrically transparent setting with nine DMUs and a single input, designed to illustrate how the margin resolves ties among CCR-efficient units and to visualise how scores evolve as $\alpha$ increases.
The second mimics a high-dimensional setting in which the CCR model saturates entirely, and serves to demonstrate that MM-DEA can provide a complete and stable ranking even in this extreme case.

\subsection{Experiment 1: Single-Input Setting}

Table~\ref{tab:data1} presents a dataset of nine DMUs (A--I), each consuming one unit of input
and producing two outputs.
The CCR model identifies four units---A, C, H, and I---as efficient, all receiving a score of 1.0000;
the remaining five are evaluated below unity (Table~\ref{tab:data1}).
Since all four efficient units are assigned the same score, the CCR model cannot rank them.

\begin{table}[ht]
\centering
\caption{Synthetic dataset 1: inputs, outputs, and CCR efficiency scores}
\label{tab:data1}
\begin{tabular}{l|ccccccccc}
\hline
DMU          & A & B & C & D & E & F & G & H & I \\
\hline
Input $x$    &  1 &  1 &  1 &  1 &  1 &  1 &  1 &  1 &  1 \\
\hline
Output $y_1$ &  1 &  2 &  3 &  4 &  4 &  5 &  6 &  7 &  8 \\
Output $y_2$ & 10 &  6 &  9 &  7 &  5 &  3 &  5 &  6 &  2 \\
\hline
CCR          & 1.0000 & 0.6667 & 1.0000 & 0.8889 & 0.7111 & 0.6765 & 0.8529 & 1.0000 & 1.0000 \\
\hline
\end{tabular}
\end{table}

Table~\ref{tab:scores1} reports the MM-DEA scores for several values of $\alpha$.
At $\alpha=0$ the scores coincide with those of the CCR model, in accordance with
Proposition~\ref{prop:alpha0}.
As $\alpha$ increases from zero, the scores decrease monotonically for every DMU, and
the four previously tied efficient units receive distinct values.
At $\alpha=0.1$, for instance, they are ranked $\mathrm{H}>\mathrm{I}>\mathrm{A}>\mathrm{C}$
with scores $0.9152$, $0.9125$, $0.9100$, and $0.9030$, respectively.
This ordering reflects the competitive lead each unit maintains over the remaining DMUs---a
quantity that the CCR model, by construction, cannot capture.\\

Figure~\ref{fig:exp1} provides a visual summary of these results.
Panel~(a) traces the MM-DEA score as a function of $\alpha$ for the four CCR-efficient units
together with two representative inefficient units.
The efficient units (A, C, H, I) start at $1.0000$ when $\alpha=0$ and decrease at differing
rates, illustrating how the margin can distinguish competitive positions that are indistinguishable to the CCR model.
Panel~(b) compares the CCR and MM-DEA ($\alpha=0.1$) scores side by side; the hatched bars
highlight the four CCR-efficient units, whose scores are clearly differentiated by the MM-DEA model.

\begin{figure}[ht]
\centering
\includegraphics[width=\textwidth]{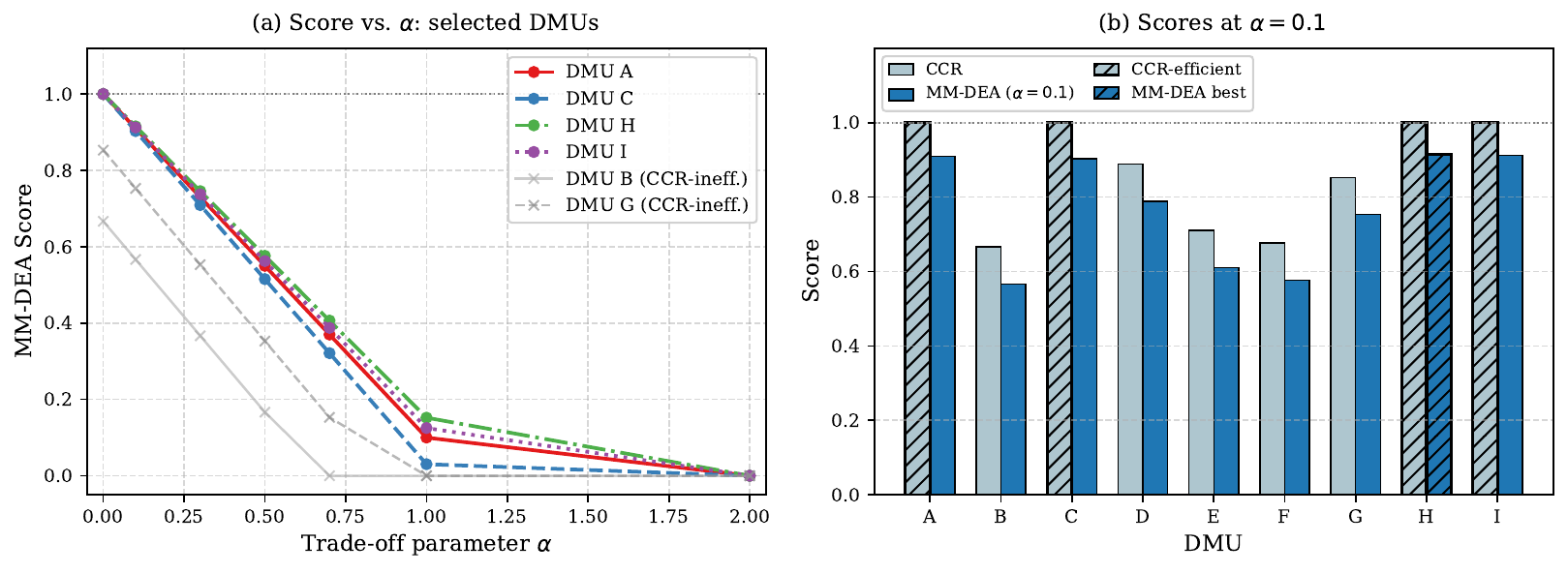}
\caption{Experiment 1: (a)~MM-DEA score as a function of~$\alpha$ for four CCR-efficient units (A, C, H, I) and two CCR-inefficient units (B, G); (b)~side-by-side CCR and MM-DEA scores at $\alpha=0.1$ (hatched bars: CCR-efficient units; single hatched bar: MM-DEA best performer).}
\label{fig:exp1}
\end{figure}

\begin{table}[ht]
\centering
\caption{CCR and MM-DEA scores --- Experiment 1}
\label{tab:scores1}
\begin{tabular}{l|ccccccccc}
\hline
DMU          & A & B & C & D & E & F & G & H & I \\
\hline
CCR          & 1.0000 & 0.6667 & 1.0000 & 0.8889 & 0.7111 & 0.6765 & 0.8529 & 1.0000 & 1.0000 \\
\hline
$\alpha=0$   & 1.0000 & 0.6667 & 1.0000 & 0.8889 & 0.7111 & 0.6765 & 0.8529 & 1.0000 & 1.0000 \\
$\alpha=0.1$ & 0.9100 & 0.5667 & 0.9030 & 0.7889 & 0.6111 & 0.5765 & 0.7529 & 0.9152 & 0.9125 \\
$\alpha=0.3$ & 0.7300 & 0.3667 & 0.7091 & 0.5889 & 0.4111 & 0.3765 & 0.5529 & 0.7457 & 0.7375 \\
$\alpha=0.5$ & 0.5500 & 0.1667 & 0.5152 & 0.3889 & 0.2111 & 0.1765 & 0.3529 & 0.5761 & 0.5625 \\
$\alpha=0.7$ & 0.3700 & 0.0000 & 0.3212 & 0.1889 & 0.0111 & 0.0000 & 0.1529 & 0.4065 & 0.3875 \\
$\alpha=1.0$ & 0.1000 & 0.0000 & 0.0303 & 0.0000 & 0.0000 & 0.0000 & 0.0000 & 0.1522 & 0.1250 \\
$\alpha=2.0$ & 0.0000 & 0.0000 & 0.0000 & 0.0000 & 0.0000 & 0.0000 & 0.0000 & 0.0000 & 0.0000 \\
\hline
\end{tabular}
\end{table}

\subsection{Experiment 2: High-Dimensional Setting}

To examine the model's behaviour in a higher-dimensional setting, we consider ten DMUs with
four inputs and four outputs each (Table~\ref{tab:data2}).
When the number of input and output variables is large relative to the number of DMUs, the CCR model suffers from a lack of discriminatory power. In this example, all ten DMUs are evaluated as CCR-efficient with a score of 1.0000 (Table~\ref{tab:data2}), and therefore the CCR model fails to provide any meaningful ranking among them.

\begin{table}[ht]
\centering
\caption{Synthetic dataset 2: inputs, outputs, and efficiency scores}
\label{tab:data2}
\resizebox{\textwidth}{!}{%
\begin{tabular}{l|cccccccccc}
\hline
DMU          & DMU1 & DMU2 & DMU3 & DMU4 & DMU5 & DMU6 & DMU7 & DMU8 & DMU9 & DMU10 \\
\hline
Input 1  &  5 &  3 &  2 &  8 & 10 &  4 &  9 &  1 &  8 &  3 \\
Input 2  &  5 &  8 &  2 &  3 &  6 &  8 &  3 &  4 &  2 & 10 \\
Input 3  &  7 &  7 & 10 &  2 &  8 &  7 &  5 &  5 &  7 &  2 \\
Input 4  &  2 &  7 &  7 &  8 &  2 &  8 &  4 &  9 &  7 &  7 \\
\hline
Output 1 &  4 &  1 &  2 &  7 & 10 &  5 &  7 &  5 &  7 &  8 \\
Output 2 &  2 &  9 & 10 &  6 &  3 &  9 &  8 &  9 &  4 &  1 \\
Output 3 &  7 &  7 &  3 &  3 &  1 &  9 &  2 &  4 &  5 &  1 \\
Output 4 &  2 &  1 &  6 &  1 &  9 &  9 &  9 &  1 &  8 &  5 \\
\hline
CCR          & 1.0000 & 1.0000 & 1.0000 & 1.0000 & 1.0000 & 1.0000 & 1.0000 & 1.0000 & 1.0000 & 1.0000 \\
\hline
$\alpha=0$   & 1.0000 & 1.0000 & 1.0000 & 1.0000 & 1.0000 & 1.0000 & 1.0000 & 1.0000 & 1.0000 & 1.0000 \\
$\alpha=0.1$ & 0.9679 & 0.9075 & 0.9601 & 0.9540 & 0.9627 & 0.9471 & 0.9452 & 0.9680 & 0.9518 & 0.9610 \\
\hline
\end{tabular}%
}
\end{table}
\vspace{2ex}

Applying MM-DEA with $\alpha=0.1$ yields ten distinct scores and thus a complete ranking
of the DMUs (Table~\ref{tab:data2}).
The results are consistent with the model's role as a generalisation of the CCR model:
at $\alpha=0$ all scores equal 1.0000, consistent with Proposition~\ref{prop:alpha0}, while in this example a small positive $\alpha$ already suffices to resolve all ties.\\

Figure~\ref{fig:exp2} visualises the outcome for Experiment~2.
Panel~(a) traces the MM-DEA score as a function of $\alpha$ for five representative DMUs.
The gap between top and bottom performers widens as $\alpha$ grows, demonstrating that
a larger margin parameter imposes a progressively stricter efficiency evaluation and thereby
amplifies the discriminatory power of the model.
Panel~(b) displays the ten DMUs sorted in descending order of their MM-DEA scores at $\alpha=0.1$,
revealing a clear hierarchy that is entirely invisible to the CCR model.

\begin{figure}[ht]
\centering
\includegraphics[width=\textwidth]{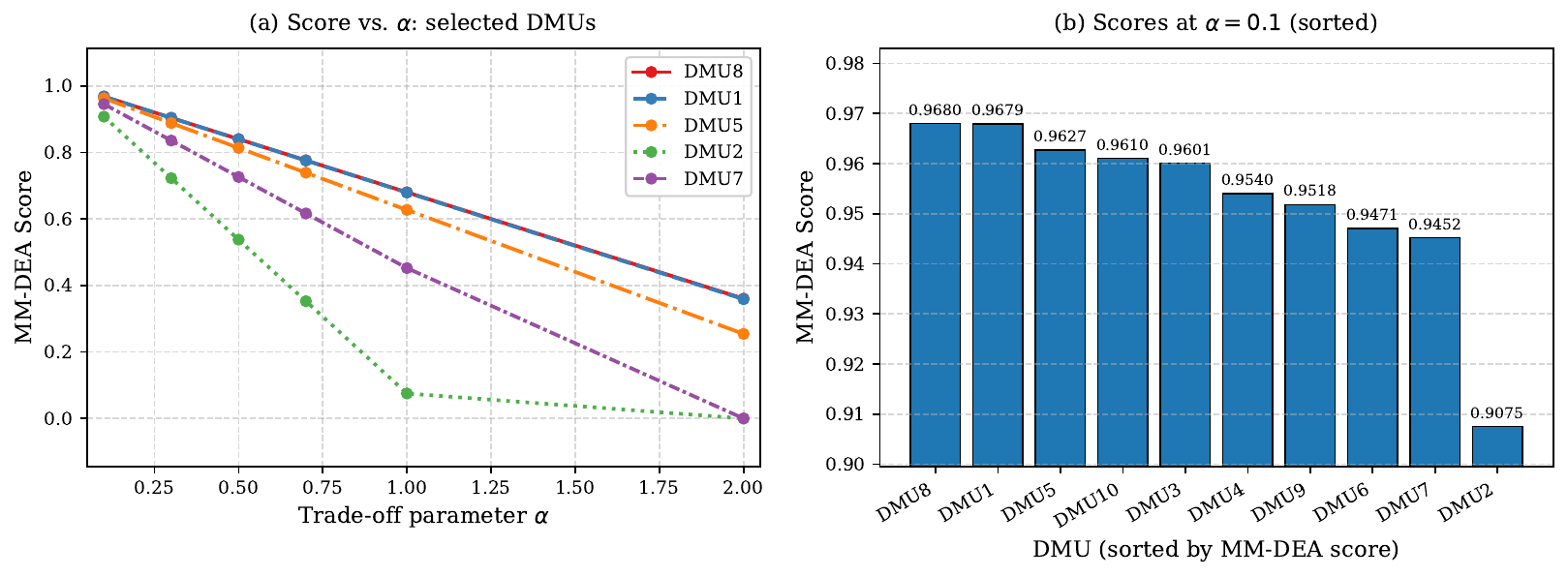}
\caption{Experiment 2: (a)~MM-DEA score as a function of~$\alpha$ for five selected DMUs; (b)~MM-DEA scores at $\alpha=0.1$ sorted in descending order (all CCR scores equal 1.0000).}
\label{fig:exp2}
\end{figure}

\subsection{Discussion}

The results from both synthetic experiments illustrate how the MM-DEA model can alleviate the ``saturation'' problem of the CCR efficiency frontier.
By incorporating the margin maximisation term, the model provides a more nuanced evaluation that rewards DMUs not just for being efficient, but for being \emph{robustly} efficient---that is, for maintaining a significant competitive lead over the reference set.

A notable observation in these two examples is that even a small value of $\alpha$ (e.g., $\alpha=0.1$) was enough to break all ties among CCR-efficient units in both datasets, while preserving the relative ordering of CCR-inefficient units.
For applications requiring a finer sensitivity to competitive margins, larger values of $\alpha$ may be employed, at the cost of reducing the absolute score for all DMUs.

\section{Conclusion}
In this paper, we proposed the Maximum-Margin DEA (MM-DEA) model as a novel framework to enhance the discriminatory power of efficiency analysis. By integrating the concept of ``margin maximization''—a core principle in Structural Risk Minimization—directly into the DEA objective function, we provided a mechanism that can rank Decision Making Units (DMUs) that are indistinguishable under the standard CCR model.\\

Our theoretical analysis demonstrated that the MM-DEA model is a generalized extension of the CCR model, reducing to it when the trade-off parameter $\alpha$ is set to zero. Furthermore, we showed that the initial fractional programming formulation can be rigorously transformed into a linear programming problem through a specialized variable transformation, ensuring computational efficiency for practical applications.\\

The numerical experiments on two synthetic datasets illustrate the potential of the proposed model. In both examples, MM-DEA differentiated among CCR-efficient units by rewarding those that maintain a significant competitive lead (margin) over their peers. This provides decision-makers with a ranking tool whose scores retain the familiar $[0,1]$ interpretation of the CCR model, in contrast to the super-efficiency model, whose scores can exceed one.\\

Despite these contributions, several avenues for future research remain. First, the systematic determination of the optimal trade-off parameter $\alpha$ remains an open question; future work could explore data-driven approaches, such as cross-validation, to select $\alpha$ based on the specific characteristics of the dataset. Second, while this study focused on the CCR model, the margin maximization framework could be extended to other DEA variants, including the BCC model for variable returns to scale or Slack-Based Measure (SBM) models. Finally, a more rigorous theoretical comparison between MM-DEA and Andersen-Petersen super-efficiency models would further clarify the mathematical conditions under which these rankings converge or diverge.\\

In conclusion, the MM-DEA model represents a significant step toward more precise efficiency measurement, bridging the gap between mathematical optimization and robust managerial decision-making.


\begin{thebibliography}{99}
\bibitem{ap} Andersen, P., and Petersen, N.C. 
A Procedure for Ranking Efficient Units in Data Envelopment Analysis.
{\it Management Science}, 39, 1261-1264 (1993). 
%
\bibitem{ccr} Charnes, A., Cooper, W.W., and Rhodes, E. 
Measuring the Efficiency of Decision Making Units. 
{\it European Journal of Operations Research}, 2, 429-444 (1978).
\bibitem{text} Cooper, W.W., Seiford, L.M., and Tone, K. 
A Comprehensive Text with Models, Applications, References and DEA-Solver Software. Springer New York, NY (2007).
%
\bibitem{dh}
Dul\'{a}, J.H. and Hickman, B.L. 
Effects of excluding the column being scored from the DEA envelopment LP technology matrix.
{\it Journal of the Operational Research Society}, 48, 1001-1012 (2017).
%
\bibitem{kt} Kitahara, T. and Tsuchiya, T. Enhancing Top Efficiency by Minimizing Second-Best Scores: A Novel Perspective on Super Efficiency Models in DEA, arXiv:2411.00438v1  (2024).
%
\bibitem{maj} Mehrabian, S., Alirezaee, M.R. and Jahanshahloo, G.R. 
A Complete Efficiency Ranking of Decision Making Units in Data Envelopment Analysis. 
{\it Computational Optimization and Applications} 14, 261–266 (1999).
%
\bibitem{mt} Ueda, H., Development and Analysis of a Maximum-Margin DEA Model for Refining Efficiency Scores (in Japanese). 
Master thesis, Joint Graduate School of Mathematics for Innovation, Kyushu University (2026).
\end{thebibliography}
\end{document}